\documentclass[leqno,11pt]{scrartcl}
\usepackage[utf8]{inputenc}

\usepackage{amsmath,amssymb,amsthm}
\usepackage{amscd}
\usepackage{setspace} 
\usepackage{enumerate}
\usepackage{array}
\usepackage{tabularx}
\usepackage{multirow}
\usepackage{booktabs}
\usepackage{url}
\usepackage{ulem}

\theoremstyle{definition}
\newtheorem{definition}{Definition}

\newtheorem{examp}{Example}

\theoremstyle{plain}
\newtheorem{Theorem}{Theorem}

\newtheorem{Lemma}[definition]{Lemma}
\newtheorem{corollary}{Corollary}

\theoremstyle{remark}

\newcommand{\C}{\mathbb{C}}

\renewcommand{\H}{\mathbb{H}}
\newcommand{\K}{\mathbb{K}}
\newcommand{\N}{\mathbb{N}}

\newcommand{\Q}{\mathbb{Q}}
\newcommand{\R}{\mathbb{R}}
\newcommand{\Z}{\mathbb{Z}}

\newcommand{\oh}{\mathcal{\scriptstyle{O}}}

\let\leq\leqslant
\let\geq\geqslant

\begin{document}

\begin{center}
\begin{huge}
\begin{spacing}{1.0}
\textbf{Hermitian theta series and Maa\ss\;spaces under the action of the maximal discrete extension of the Hermitian modular group}  
\end{spacing}
\end{huge}

\bigskip
by
\bigskip

\begin{large}
\textbf{Annalena Wernz\footnote{Lehrstuhl A für Mathematik, RWTH Aachen University, D-52056 Aachen, Germany \\ annalena.wernz@rwth-aachen.de \\
\hspace*{-1ex}Annalena Wernz was partially supported by Graduiertenkolleg Experimentelle und konstruktive Algebra at RWTH Aachen University.}}
\end{large}
\vspace{0.5cm}\\
\vspace{1cm}
\end{center}
\noindent\textbf{Abstract.}
Let $\Gamma_n(\oh_{\K})$ denote the Hermitian modular group of degree $n$ over an imaginary quadratic number field $\K$ and  $\Delta_{n,\K}^*$ its maximal discrete extension in the special unitary group $SU(n,n;\C)$. In this paper we study the action of $\Delta_{n,\K}^*$ on Hermitian theta series and Maa\ss\;spaces. For $n=2$ we will find theta lattices such that the corresponding theta series are modular forms with respect to $\Delta_{2,\K}^*$ as well as examples where this is not the case. Our second focus lies on studying two different Maa\ss\;spaces. We will see that the new found group $\Delta_{2,\K}^*$ consolidates the different definitions of the spaces.
\medskip

\vspace{2ex}

\newpage

\noindent\textbf{1. Introduction.} In 1950, Hel Braun introduced the Hermitian modular group $\Gamma_n(\oh_{\K})$ as a generalization of the elliptic modular group \cite{braun1}. Its maximal discrete extension $\Delta_{n,\K}^*$ in $SU(n,n;\C)$, called the extended Hermitian modular group, was determined in \cite{kriegraumwernz}. We study the action of this extended group on theta series $\Theta(Z,\Lambda)^{(n)}$ where $Z$ is an element of the Hermitian half space $\H_n$ and $\Lambda$ is a theta lattice. We will find a sufficient condition such that $\Theta(Z,\Lambda)^{(n)}$ is a modular form with respect to $\Delta_{n,\K}^*$ for all $n\in \N$ and a less strict condition such that it is a modular form for a fixed $n\in \N$. For $n=2$ we consider specific theta lattices for which the corresponding theta series is a modular form with respect to $\Delta_{2,\K}^*$ as well as examples where this is not the case. Furthermore, we study Maa\ss\;spaces introduced by Sugano in 1985 \cite{sugano} and Krieg in 1991 \cite{krieg}. It is known that Krieg's Maa\ss\;space is always a subspace of Suganos' Maa\ss\;space. Moreover, there are known cases for which they coincide as well as for which they differ. By studying the extended Hermitian modular group once more, we will prove that they coincide if and only if all elements in Sugano's Maa\ss\;space are modular forms with respect to $\Delta_{n,\K}^*$.\\

\noindent\textbf{2. Hermitian modular forms.} For a matrix $M$ let $\overline{M}$ and $M^{tr}$ denote the complex conjugate and transposed of $M$, respectively. We write $M>0$ if $M$ is positive definite and $M\geq0$ if $M$ is positive semidefinite. Define the \textit{special unitary group} $SU(n,n;\C)$ as the set of all matrices
\begin{align}\label{text1}
 M=\begin{pmatrix}A&B\\C&D\end{pmatrix} \in SL_{2n}(\C)\quad \text{with}\quad J[M]:=\overline{M}^{tr}JM=J,
\end{align}
where $J=\left(\begin{smallmatrix}0&-I\\I&0\end{smallmatrix}\right)$ and $I$ denotes the $n\times n$-identity matrix. Then $SU(n,n;\C)$ acts on the Hermitian half space
\begin{align*}
 \H_n:=\left\{Z\in \C^{n\times n},\; \frac{1}{2i}(Z-\overline{Z}^{tr}) >0\right\}
\end{align*}
via
\begin{align*}
 (M,Z)\mapsto M\langle Z\rangle :=(AZ+B)(CZ+D)^{-1}
\end{align*}
with $M=\left(\begin{smallmatrix}A&B\\C&D\end{smallmatrix}\right)$ as in (\ref{text1}). For $m\in \N$, $m$ squarefree, we consider the imaginary quadratic number field $\K:=\Q(\sqrt{-m})\subset \C$ with discriminant 
\begin{align*}
 d_{\K}=\begin{cases} -4m,\quad &m\not\equiv3 \;(\text{mod }4),\\
         -m,\quad &m\equiv3 \;(\text{mod }4).
        \end{cases}
\end{align*}
Its ring of integers is given by 
\begin{align*}
 \oh_{\K}=\Z+\omega_{\K}\Z=\begin{cases} \Z+\sqrt{-m}\Z,\quad &m\not\equiv3 \;(\text{mod }4),\\
         \Z+\frac{1+\sqrt{-m}}{2}\Z,\quad &m\equiv3 \;(\text{mod }4).
        \end{cases}
\end{align*}
By $\oh_{\K}^*$ we denote its unit group. The \textit{Hermitian modular group} of degree $n$ is
\begin{align*}
 \Gamma_n(\oh_{\K}):=SU(n,n;\C)\cap \oh_{\K}^{2n\times 2n},
\end{align*}
where $\Gamma_1(\oh_{\K})=SL_2(\Z)$. By \cite{kriegraumwernz} it is known that its maximal discrete extension in $SU(n,n;\C)$ is given by 
\begin{align*}
 \Delta_{n,\K}^*=\bigcup_{A \in \mathcal{A}}M_A \Gamma_n(\oh_{\K}),
\end{align*}
where 
\begin{align}\label{text2}
M_A=\left(\begin{smallmatrix}A&0\\0&\overline{A}^{-tr}\end{smallmatrix}\right)\hspace{-0.1cm},\,
\mathcal{A}=\left\{\hspace{-0.05cm}A=\frac{1}{u}L\in SL_n(\C);\, u\in \C,\, L\in \oh_{\K}^{n \times n}, \, \mathcal{I}(L)^n=\oh_{\K}\det(L)\hspace{-0.05cm}\right\}
\end{align}
and $\mathcal{I}(L)$ denotes the ideal generated by the entries of $L$. Particularly, $\mathcal{I}(L)^n=\oh_{\K}\det(L)$ already implies
\begin{align*}
 N(\mathcal{I}(L))^n=N(\mathcal{I}(L)^n)=N(\oh_{\K}\det(L))=N(\oh_{\K}u^n)=(u\overline{u})^n,
\end{align*}where $N(\mathcal{I})$ denotes the norm of an ideal $\mathcal{I}$ in $\oh_{\K}$. A holomorphic function $f:\H_n\to \C$ is called a \textit{Hermitian modular form} of degree $n$ and weight $r\in \Z$ if
\begin{align*}
 f(M\langle Z\rangle)=\det(CZ+D)^{-r}f(Z)
\end{align*}
for all $Z\in \H_n$ and $M\in \Gamma_n(\oh_{\K})$ and $f$ is bounded on $\{Z\in \H_1,\,\Im(Z)\geq1\}$ if $n=1$. The set of all Hermitian modular forms of degree $n$ and weight $r$ forms a finite dimensional vector space which we denote by $[\Gamma_n(\oh_{\K}),r]$. We define $[\Delta_{n,\K}^*,r]$ analogously.
\mbox{}\\

\noindent\textbf{3. Theta series.} In the following we study theta series. Let $V\subseteq \C^r$ be an $r$-dimensional $\K$-vector space with the standard Hermitian form
\begin{align*}
 h:\C^r\times \C^r\to \C,\;(x,y)\mapsto \overline{x}^{tr}y.
\end{align*}
We call a set $\Lambda\subseteq V$ an \textit{$\oh_{\K}$-lattice} of rank $r$ if $\Lambda$ is an $\oh_{\K}$-module and a $\Z$-lattice of rank $2r$. We consider
\begin{align*}
 F_h:\C^r\times \C^r \to \R,\,(x,y)\mapsto F_h(x,y)=\Re(h(x,y)).
\end{align*}
Then $(\Lambda,F_h)$ is a positive definite $\Z$-lattice. If $(\Lambda,F_h)$ is even and unimodular, then $(\Lambda,h)$ is called a \textit{theta lattice} of rank $r$. Let $(\Lambda, h)$ be a theta lattice. Then by \cite{cohenresnikoff} it is known that the associated \textit{theta series}
\begin{align*}
 \Theta(Z,\Lambda)^{(n)}:=\sum_{T\geq0}\#(\Lambda,T)e^{i\pi tr(TZ)}
\end{align*}
with Fourier coefficients
\begin{align*}
 \#(\Lambda,T)=\#\{\lambda=(\lambda_1,\ldots,\lambda_n)\in \Lambda^n,\;\overline{\lambda}^{tr}\lambda=T\}\in \N_0
\end{align*}
is a Hermitian modular form of degree $n$ and weight $r$. Clearly the theta series associated to theta lattices $\Lambda$ and $\Lambda'$ coincide whenever $\Lambda$ and $\Lambda'$ are isometric. Considering
\begin{align*}
 \Theta(Z^{tr},\Lambda)^{(n)}=\Theta(Z,\overline{\Lambda})^{(n)},
\end{align*}
where $\overline{\Lambda}$ denotes the complex conjugate of $\Lambda$, we see that applying an automorphism of the Hermitian half space can yield a theta series with respect to an altered theta lattice. In the following, we study $\Lambda$ under the action of $\Delta_{n,\K}^*$.

\begin{Theorem}\label{Theorem1}
Let $(\Lambda,h)$ be a theta lattice of rank $r$ and $M\in \Delta_{n,\K}^*$ of the form (\ref{text2}). Then 
\begin{align*}
\Theta(M\langle Z\rangle,\Lambda)^{(n)}=\Theta(Z,\Lambda')^{(n)}, 
\end{align*}
where $\Lambda'=\frac{1}{\sqrt{N(\mathcal{I}(L))}}\mathcal{I}(L)\Lambda$. Moreover, $\Lambda'$ is a theta lattice.
\end{Theorem}

Particularly, the theta series coincide whenever $\Lambda$ and $\Lambda'$ are isometric.

\begin{proof}
 As $\Lambda'$ is an $\oh_{\K}$-module by definition and we have
 \begin{align*}
  \sqrt{N(\mathcal{I}(L))}\Lambda \subseteq \Lambda'\subseteq\frac{1}{\sqrt{N(\mathcal{I}(L))}}\Lambda,
 \end{align*}
$\Lambda'$ is an $\oh_{\K}$-lattice of rank $r$. That $(\Lambda',F_h)$ is even can be shown by considering $F_h(\lambda',\lambda')$ with arbitrary $\lambda'=\frac{1}{u}\sum_j l_j\lambda_j$ with $l_j\in \mathcal{I}(L)$ and $\lambda_j\in\Lambda$ and using $\overline{l_j}l_k\in N(\mathcal{I}(L))\oh_{\K}$. Denote by $\Lambda'^{\#}$ the dual lattice of $\Lambda'$, then $\Lambda'^{\#}=\Lambda'$ follows directly from the definition and the fact that $\Lambda$ is unimodular. Hence, $\Lambda'$ is a theta lattice.

Let $\Lambda_1'\times \ldots \times \Lambda_n':=\Lambda^n A$. Then $\Lambda_j'$ does not depend on the representative $A$ as a multiplication by a unimodular matrix does not change the lattice. If $i \neq j \in \{1,\ldots,n\}$, let $U$ denote the permutation matrix of $i$ and $j$. Then $AU$ is a representative of the same coset as $A$ because the modular group ist normal in $\Delta_{n,\K}^*$ due to \cite{kriegraumwernz}. We obtain $\Lambda_j'=\Lambda_i'$ and hence $\Lambda_1'=\ldots=\Lambda_n'$. By definition of $\mathcal{I}(L)$ we immediately obtain
\begin{align*}
 \Lambda_1'=\ldots=\Lambda_n'=\frac{1}{u}\mathcal{I}(L)\Lambda=:\Lambda^*
\end{align*}
with $u\in \C$ as in (\ref{text2}). Now let $Z\in \H_n$. We consider 
\begin{align*}
 \Theta(M\langle Z\rangle,\Lambda)^{(n)}=\sum_{T\geq 0}\#(\Lambda,T)e^{i\pi tr(TM\langle Z \rangle)}=\sum_{T\geq 0}\#(\Lambda,T)e^{i\pi tr(T^*Z)}
\end{align*}
with $T^*=\overline{A}^{tr}TA\geq0$. Since 
\begin{align*}
 \Lambda^n\to {\Lambda^*}^n,\,(\lambda_1,\ldots,\lambda_n)\mapsto (\lambda_1,\ldots,\lambda_n)A
\end{align*}
is a bijection, we obtain $\#(\Lambda,T)=\#(\Lambda^*,T^*)$ and thus
\begin{align*}
 \Theta(M\langle Z\rangle,\Lambda)^{(n)}=\sum_{T\geq 0}\#(\Lambda,T)e^{i\pi tr(T^*Z)}=\sum_{T\geq 0}\#(\Lambda^*,T^*)e^{i\pi tr(T^*Z)}= \Theta(Z,\Lambda^*)^{(n)}.
\end{align*}
As $\Lambda^*=\frac{1}{u}\mathcal{I}(L)\Lambda \cong \frac{1}{\sqrt{N(\mathcal{I}(L))}}\mathcal{I}(L)\Lambda=\Lambda'$, that is, $\Lambda^*$ and $\Lambda'$ are isometric, we obtain
\begin{align*}
  \Theta(M\langle Z\rangle,\Lambda)^{(n)}=\Theta(Z,\Lambda')^{(n)}.
\end{align*}
\end{proof} 

In light of Theorem \ref{Theorem1} and based on the definition in \cite{quebbemann} we call a theta lattice $(\Lambda,h)$ \textit{strongly modular} if for all integral ideals $\mathcal{I}\leq \oh_{\K}$ we have $\Lambda\cong \frac{1}{\sqrt{N(\mathcal{I})}}\mathcal{I}\Lambda$. The following Theorem follows directly from this definition and Theorem \ref{Theorem1}.

\begin{Theorem}\label{Theorem2}
 If a theta lattice $\Lambda$ of rank $r$ is strongly modular, the corresponding theta series $\Theta(Z,\Lambda)^{(n)}$ is a modular form of weight $r$ with respect to the maximal discrete extension $\Delta_{n,\K}^*$ for all $n\in \N$.
\end{Theorem}

By applying Theorem 1 from \cite{kriegraumwernz}, we obtain the following corollary for fixed $n\in \N$.

\begin{corollary}\label{Corollary1}
Let $n\in \N$. Then $\Theta(Z,\Lambda)^{(n)}$ is a modular form with respect to $\Delta_{n,\K}^*$ if $\frac{1}{\sqrt{N(\mathcal{I})}}\mathcal{I}\Lambda\cong \Lambda$ for all integral ideals $\mathcal{I}$ whose order in the class group is a divisor of $n$. 
\end{corollary}

Strong modularity is a restrictive requirement but allows us to make a statement for arbitrary degree. As seen in Theorem \ref{Theorem1}, strong modularity is not a necessary condition for $\Theta(Z,\Lambda)^{(n)}$ to be a modular form with respect to $\Delta_{n,\K}^*$. Isometry is at most necessary for all integral ideals which can be generated by matrices $L$ of the form (\ref{text2}). For fixed $n\in \N$ we know from \cite{kriegraumwernz} that these are precisely those ideals in $\oh_{\K}$ whose $n$-th power is a principal ideal, as was formulated in Corollary \ref{Corollary1}. These can be constructed explicitely for $n=2$. In this case we have
\begin{align*}
 \Delta_{2,\K}^*=\bigcup_{\substack{d|d_{\K}\\ d \;\Box-\text{free}}}\Gamma_2(\oh_{\K})W_d
\end{align*}
with $W_d=\left(\begin{smallmatrix}V_d&0\\0&\overline{V_d}^{-tr}\end{smallmatrix}\right)$, where $V_d$ denote the Atkin-Lehner involutions
\begin{align*}
 V_d=\frac{1}{\sqrt{d}}\begin{pmatrix}\alpha d&\beta (m+\sqrt{-m})\\ \gamma(m-\sqrt{-m})&\delta d\end{pmatrix}\in SL_2(\C),\;\alpha,\beta,\gamma,\delta \in \Z,
\end{align*}
\cite{wernz},  \cite{kriegraumwernz}. Because of $V_dSL_2(\oh_{\K})=SL_2(\oh_{\K})V_d$ and the form of $\Delta_{2,\K}^*$, $V_d$ does not depend on the choice of $\alpha, \beta,\gamma,\delta$ in our setting and thus is well defined. Considering the special structure of $\Delta_{2,\K}^*$, we see that in order to understand the behavior of theta lattices under the action of $\Delta_{2,\K}^*$, we only need to study finitely many lattices $\mathcal{A}_d$ corresponding to $V_d$. The unique ideal in $\oh_{\K}$ of norm $d$ is given by
\begin{align*}
 \mathcal{A}_d=\oh_{\K}d+\oh_{\K}(m+\sqrt{-m}).
\end{align*}
Hence, it is sufficient to study $\Lambda'=\frac{1}{\sqrt{d}}\mathcal{A}_d \Lambda$ for all squarefree divisors $d$ of $d_{\K}$.

We consider a specific lattice given by \cite{hentschelnebekrieg}. 
\begin{examp}
Let $d$ be a squarefree divisor of $d_{\K}$. Choose $\alpha,\beta \in \Z$ such that \newline$d+1+\alpha^2+\beta^2\equiv 0 \,(\text{mod }d_{\K})$ and define
\begin{align*}
 u:=\alpha+\beta+\sqrt{-d} \in \oh_{\K},\quad v:=\alpha-\beta+\sqrt{-d}\in \oh_{\K}.
\end{align*}
Then
\begin{align*}
 \Lambda=\oh_{\K}\left(\begin{smallmatrix}1\\1\\0\\0\end{smallmatrix}\right)+\oh_{\K}\left(\begin{smallmatrix}-1\\1\\0\\0\end{smallmatrix}\right)+\oh_{\K}\frac{1}{\sqrt{d_{\K}}}\left(\begin{smallmatrix}u\\v\\1\\1\end{smallmatrix}\right)+\oh_{\K}\frac{1}{\sqrt{d_{\K}}}\left(\begin{smallmatrix}-\overline{v}\\\overline{u}\\-1\\1\end{smallmatrix}\right)
\end{align*}
is a free theta lattice of rank $4$. Computations with magma yield $\frac{1}{\sqrt{d}}\mathcal{A}_d\Lambda\cong \Lambda$ for all squarefree divisors $d$ of $d_{\K}$ for 
\begin{align*}
 m\in \{&1,2,3,5,6,7,10,11,13,14,15,17,19,21,22,23,26,29,31,34,35,37,38,39,41,43,\\
 &46,47,53,55,58,59,61,62,65,67,70,71,73,74,79,82,83,86,89,91,93,94,95,97\}.
\end{align*}
Hence, $\Theta(Z,\Lambda)^{(2)}$ is a modular form of degree $2$ and weight $4$ with respect to $\Delta_{2,\K}^*$ in these cases. By comparing the Fourier expansions, we obtain $\Theta(Z,\Lambda)^{(2)}\neq \Theta(Z,\Lambda')^{(2)}$ for all other admissible $m\leq 100$. These are given by
\begin{align*}
 m\in\{30,33,42,51,57,66,69,77,78,85,87\}.
\end{align*}
In particular, we see that whenever a theta series is a modular form with respect to $\Delta_{2,\K}^*$ for squarefree $m\leq 100$, we immediately have $\Lambda\cong \frac{1}{\sqrt{d}}\mathcal{A}_d \Lambda$ for all squarefree divisors $d$ of $d_{\K}.$
\end{examp}

\noindent\textbf{4. The Maa\ss\;Spaces.} Let $\oh_{\K}^{\star}:=\frac{1}{\sqrt{d_{\K}}}\oh_{\K}$. By \cite{braun2}, \cite{krieg} it is known that any Hermitian modular form of degree $2$ and weight $r\in \Z$ has a Fourier expansion
\begin{align*}
 f(Z)=\sum_{\substack{T\in \Lambda(2,\oh_{\K})\\ T\geq 0}}\alpha_f(T)e^{2\pi i tr(TZ)}, \quad Z\in \H_2,
\end{align*}
with Fourier coefficients $\alpha_f(T)$ and
\begin{align*}
 \Lambda(2,\oh_{\K})=\{T=(t_{\nu \mu})_{\nu,\mu}\in \K^{2\times 2}, \; \overline{T}^{tr}=T, \, t_{\nu\nu}\in \Z,\, t_{\nu\mu}\in \oh_{\K}^{\star}\, \text{for }\mu\neq \nu\}.
\end{align*}
For $T\neq 0$ let $\epsilon(T):=\max\left\{q\in \N ,\; \frac{1}{q}T\in \Lambda(2;\oh_{\K})\right\}$. The Maa\ss\;space by Sugano \cite{sugano}, denoted by $\mathcal{S}(r,\oh_{\K})$, consists of all $f\in[\Gamma_2(\oh_{\K}),r]$ whose Fourier coefficients satisfy 
\begin{align}\label{FourierSugano}
 \alpha_f(T)=\sum_{\substack{\eta \in \N \\ \eta|\epsilon(T)}}\eta^{r-1}\alpha_f\left(\begin{smallmatrix}1&t/\eta\\\overline{t}/\eta &lk/\eta^2\end{smallmatrix}\right)
\end{align}
for all $T=\left(\begin{smallmatrix}k&t\\\overline{t}&l\end{smallmatrix}\right)\in \Lambda(2,\oh_{\K})$, $T\geq 0$, $T\neq 0$. In \cite{krieg}, Krieg defines another Maa\ss\;space, denoted by $\mathcal{M}(\oh_{r,\K})$, of all $f\in[\Gamma_2(\oh_{\K}),r]$ for which there is an $\alpha_f^*:\N_0\to \C$ such that the Fourier coefficients of $f$ satisfy
\begin{align}\label{text:FourierKrieg}
 \alpha_f(T)=\sum_{\substack{\eta \in \N \\ \eta|\epsilon(T)}}\eta^{r-1}\alpha_f^*(-d_{\K}\det(T)/\eta^2)
\end{align}
for all $T\in \Lambda(2,\oh_{\K})$, $T\geq 0$, $T\neq 0$. One immediately sees that Krieg's Maa\ss\;space is contained in Sugano's Maa\ss\;space. Furthermore, the Maa\ss\;spaces coincide whenever $d_{\K}$ is a prime discriminant.

\begin{Lemma}\label{Lemma1}
 For $d|d_{\K}$, $d$ squarefree and $T\in \C^{2\times 2}$ we have 
 \begin{align*}
  T\in \Lambda(2,\oh_{\K}) \quad \Leftrightarrow \quad T':=T[V_d]\in  \Lambda(2,\oh_{\K}).
 \end{align*}
Furthermore, if $T\in  \Lambda(2,\oh_{\K}) $ we obtain $\epsilon(T)=\epsilon(T').$
\end{Lemma}
\begin{proof}
The proof is very intuitive when applying the isomorphism $\Phi$ between $SU(2,2;\C)$ and the orthogonal group $O(2,4)$ in chapter $3$ of \cite{wernz} and Theorem $3$ of \cite{kriegraumwernz}, respectively, and considering the orthogonal instead of the Hermitian setting. Here, the matrix $T$ becomes the vector $\left(\begin{smallmatrix}k&v&w&l\end{smallmatrix}\right)^{tr}\in \R^4$ with $v+\omega_{\K} w=t$ and the action of $V_d$ becomes a multiplication by a matrix in $GL_4(\Z)$ by said isomorphism. Alternatively, Lemma \ref{Lemma1} can be verified by a simple computation.
\end{proof}

For $T=\left(\begin{smallmatrix}k&t\\\overline{t}&l\end{smallmatrix}\right) \in \Lambda(2,\oh_{\K}) $ we write $\alpha_f(k,t,l):=\alpha_f(T)$. We use this notation in the following Lemma. 

\begin{Lemma}\label{Lemma2}
 Let $r\in\Z$, $f\in \mathcal{S}(r,\oh_{\K})$ and $d|m$. Then the following are equivalent:
 \begin{enumerate}[(i)]
  \item $f(Z[\overline{V_d}^{tr}])=f(Z)$,
  \item $\alpha_f(T[V_d^{-1}])=\alpha_f(T)$ for all $T\in  \Lambda(2,\oh_{\K}) $,
  \item $\alpha_f(1,t,l)=\alpha_f(1,t',l')$ if $l-t\overline{t}=l'-t'\overline{t'}$ and
  \begin{align}\label{text3}
   \Re(t')- \Re(t) \in \Z; \quad 2\sqrt{m}\Im(t') \equiv \begin{cases} 2\sqrt{m}\Im(t) \quad &(\text{mod }2m/d), \\ -2\sqrt{m}\Im(t),\quad &(\text{mod }2d).  \end{cases}
  \end{align}
 \end{enumerate}
\end{Lemma}
\begin{proof}
 The proof is based on the proof of Theorem $2$ in \cite{heimkrieg}. The equivalence of $(i)$ and $(ii)$ follows from the uniqueness of the Fourier expansion. In order to prove the equivalence of $(ii)$ and $(iii)$ we first acknowledge the form of $\alpha_f(T)$ in (\ref{FourierSugano}), i.e. it suffices to consider $k=1$, as well as $\alpha_f(1,t+u,l')=\alpha_f(1,t,l)$ for $u \in \oh_{\K}$ and suitable $l'$. Now assume that $(iii)$ holds. We consider
 \begin{align*}
  V_d':=\frac{1}{\sqrt{d}}\begin{pmatrix}\alpha' d& \beta' \sqrt{-m} \\ -\gamma' \sqrt{-m}& \delta' d \end{pmatrix}\in SL_2(\C), \quad \alpha',\beta', \gamma',\delta'\in \Z.
 \end{align*}
For $d|m$ we have $W_d'=\left(\begin{smallmatrix}V_d'&0\\0&\overline{V_d'}^{tr}\end{smallmatrix}\right)\in \Delta_{2,\K}^*$ and $V_dV_d'\in SL_2(\oh_{\K})$ which is not necessarily true for $m \equiv 1$ $($mod $4)$ and $2|d$. Using this and a suitable shift about $u\in \oh_{\K}$ one finds a $T'$ with Fourier coefficient $\alpha_f(1,t',l')=\alpha_f(T[V_d'^{-1}])$ which satisfies (\ref{text3}). Applying $(iii)$ yields the result. Now assume that $(ii)$ holds. For $T,T' \in \Lambda(2,\oh_{\K})$ which satisfy (\ref{text3}) we find $T''$ as before, i.e. $T^{''}$ with $\alpha_f(1,t^{''},l^{''})=\alpha_f(T[V_d'^{-1}])$ which satisfies (\ref{text3}) and apply the Chinese remainder Theorem to obtain $(iii)$.
\end{proof}

With this Lemma we are now able to prove the connection between Sugano's and Krieg's Maa\ss\;spaces in the following Theorem.

\begin{Theorem}\label{Theorem3}
 A Maa\ss\;form $f\in \mathcal{S}(r,\oh_{\K})$ belongs to $\mathcal{M}(r,\oh_{\K})$ if and only if $f\in [\Delta_{2,\K}^*,r]$.
\end{Theorem}
\begin{proof}
 As for Lemma \ref{Lemma2}, this proof is based on a proof for the paramodular group in \cite{heimkrieg}.\\
 $"\Rightarrow"$: Let $f\in \mathcal{M}(r,\oh_{\K})$. We obtain $\alpha_f(T[V_d^{-1}])=\alpha_f(T)$ for all $T \in \Lambda(2,\oh_{\K})$ by considering the identity 
 \begin{align*}
\alpha_f(T)=\sum_{\substack{\eta \in \N \\ \eta|\epsilon(T)}}\eta^{r-1}\alpha_f^*(-d_{\K}\det(T)/\eta^2)  
 \end{align*}
and applying Lemma \ref{Lemma1}. \\
$"\Leftarrow"$: It suffices to show that for $f\in \mathcal{S}(r,\oh_{\K})\cap [\Delta_{2,\K}^{*},r]$
\begin{align*}
 \alpha_f(1,t,l)=\alpha_f(1,t',l') \quad\text{holds whenever}\quad \det(T)=\det(T').
\end{align*}
 Let $m\equiv 3 \;(\text{mod }4)$ and consider $T,T'$ with $\det(T)=\det(T')$. As we can shift $t,t'$ about $\oh_{\K}$ without affecting the Fourier coefficient, we can assume $\Re(t)=\Re(t')=0$. Using this and $t,t' \in \oh_{\K}^{\star}$, there are $t_0,t_0' \in \Z$ such that
\begin{align*}
 t=i\frac{1}{2\sqrt{m}}t_0, \quad t'=i\frac{1}{2\sqrt{m}}t_0'.
\end{align*}
$\det(T)=\det(T')$ then yields $t_0\equiv t_0' \;(\text{mod }2)$. We define 
\begin{align*}
 d:=\prod_{\substack{p \in \mathbb{P}\\p|m\\p|(t_0+t_0')}}p. 
\end{align*}
Then $f\in [\Delta_{2,\K}^*,r]$ and Lemma \ref{Lemma2} yield $\alpha_f(T')=\alpha_f(T)$. Now let $m \not\equiv 3 \;(\text{mod }4)$. If $\Re(t)-\Re(t')\in \Z$, we can proceed as above. Now consider $\Re(t)-\Re(t')\in \frac{1}{2}\Z\backslash\Z$. Applying $\det(T)=\det(T')$ and the form of $\oh_{\K}^{\star}$, we can infer $m\equiv 1\;\text{(mod }4)$. As can be seen in the example
\begin{align*}
 T=\begin{pmatrix} 1& \frac{1}{2\sqrt{-5}}5\\-\frac{1}{2\sqrt{-5}}5&l\end{pmatrix}, \quad T'=\begin{pmatrix}1&\frac{1}{2}\\\frac{1}{2}&l-1\end{pmatrix}
\end{align*}
for $m=5$ and $l>2$, the case $\Re(t)-\Re(t')\in \frac{1}{2}\Z\backslash \Z$ can occur. To resolve this, we consider $T'[V_2^{-1}]$. Because of $f\in [\Delta_{2,\K}^*,r]$ we obtain $\alpha_f(T'[V_2^{-1}])=\alpha_f(T')$. Furthermore, we have $\det(T)=\det(T'[V_2^{-1}])$ and $\Re(t)-\Re(t^{''})\in \Z$ where $T'[V_2^{-1}]=\left(\begin{smallmatrix} \ast&t^{''}\\ \ast&\ast\end{smallmatrix}\right)$. Hence, we can consider $T$ and $T'[V_2^{-1}]$ and obtain $\alpha_f(T)=\alpha_f(T'[V_2^{-1}])=\alpha(T')$ as in the previous cases.
\end{proof}


\bigskip
\setlength{\parindent}{0pt}
{\small \textbf{Acknowledgement.} The author thanks Aloys Krieg for the valuable input and support as well as Markus Kirschmer and Simon Eisenbarth for their support regarding magma.} 


\nocite{Koe}  
\bibliographystyle{plain}
\renewcommand{\refname}{Bibliography}

\end{document}